\documentclass[reqno,11pt]{amsart}
\usepackage{amsmath, amssymb, amsthm, graphicx}
\usepackage{stmaryrd}
\usepackage{xcolor}
\usepackage{url}

\definecolor{modColorNew}{rgb}{01,0,0}
\definecolor{modColorOld}{rgb}{0,1,0}

\font\dixmath=cmsy10
\renewcommand{\aa}[2]{a_{#1, #2}}

\newcommand{\abs}[1]{\vert#1\vert}
\newcommand{\BB}[1]{B_{#1}}
\newcommand{\bidual}{^{\hspace{-0.05em}\raise0.6pt\hbox{$\scriptscriptstyle+$}\hspace{-0.05em}*}}
\newcommand{\BKL}[1]{B_{#1}\bidual}
\newcommand{\SBKL}[1]{A_{#1}}
\newcommand{\BP}[1]{B_{#1}\plusexp}
\newcommand{\SBP}[1]{\Sigma_{#1}}
\newcommand{\br}{\beta}
\newcommand{\brbr}{\gamma}

\newcommand{\brr}{\beta'}
\newcommand{\brrr}{\beta''}

\newcommand{\compL}{f_{\hspace{-0.2em}L}^\nn}
\newcommand{\compR}{f_{\hspace{-0.2em}R}^\nn}

\newcommand{\ddd}[1]{\delta_{#1}}

\newcommand{\denL}{D_{\hspace{-0.1em}L}}
\newcommand{\denR}{D_{\hspace{-0.1em}R}}

\newcommand{\equal}\equiv
\newcommand{\ff}[1]{\phi_{\hspace{-0.5pt}\raise-1pt\hbox{$\scriptstyle #1$}}}
\newcommand{\fl}[1]{\Phi_{\hspace{-1pt}\raise-0.5pt\hbox{$\scriptstyle #1$}}}
\newcommand{\gcdL}{\wedge_{\scriptscriptstyle L}}

\renewcommand{\ge}{\geqslant}

\newcommand{\ie}{\emph{i.e.}}
\newcommand{\inv}{^{\minus\hspace{-0.1em}1}}
\newcommand{\lcmL}{\vee_{\hspace{-0.2em}\raise1pt\hbox{$\scriptscriptstyle L$}}}
\newcommand{\lcmR}{\vee_{\hspace{-0.2em}\raise1pt\hbox{$\scriptscriptstyle R$}}}
\newcommand{\Ldots}{...\,} 
\renewcommand{\le}{\leqslant}
\newcommand{\len}[1]{\vert#1\vert}
\newcommand{\lens}[1]{{\parallel}#1{\parallel}_{\sigma}}
\newcommand{\lena}[1]{{\parallel}#1{\parallel}_A}
\newcommand{\minus}{\mathchoice{-}{-}{\raise0.7pt\hbox{$\scriptscriptstyle-$}\scriptstyle}{-}}
\newcommand{\newintegeri}[2]{
	\expandafter\def\csname #1\endcsname{#2}
	\expandafter\def\csname #1o\endcsname{{#2\minus1}}
	\expandafter\def\csname #1t\endcsname{{#2\minus2}}
	\expandafter\def\csname #1p\endcsname{{#2\plus1}}
	\expandafter\def\csname #1pp\endcsname{{#2\plus2}}}
\newcommand{\newintegerii}[1]{\newintegeri{#1#1}{#1}}
\newcommand{\numL}{N_{\hspace{-0.1em}L}}
\newcommand{\numR}{N_{\hspace{-0.1em}R}}

\newcommand{\plus}{\mathchoice{+}{+}{\raise0.7pt\hbox{$\scriptscriptstyle+$}\scriptstyle}{+}}
\newcommand{\plusminus}{\mathchoice{\pm}{\pm}{\raise0.7pt\hbox{$\scriptscriptstyle\pm$}\scriptstyle}{\pm}}
\newcommand{\plusexp}{^{\raise0.8pt\hbox{$\hspace{-0.05em}\scriptscriptstyle+$}}}

\newcommand{\revR}[1]{{\curvearrowright}^{\scriptscriptstyle(#1)}_{\scriptscriptstyle R}}

\newcommand{\revvL}{\curvearrowright_{\scriptscriptstyle L}}
\newcommand{\revvR}{\curvearrowright_{\scriptscriptstyle R}}
\newcommand{\sig}[1]{\sigma_{\!#1}} 
\newcommand{\siginv}[1]{\sigma_{\!#1}^{\hspace{-0.05em}\raise0.8pt\hbox{$\scriptscriptstyle-$}\hspace{-0.1em}1}}
\newcommand{\sigpm}[1]{\sigma_{\!#1}^{\hspace{-0.05em}\raise0.8pt\hbox{$\scriptscriptstyle\pm$}\hspace{-0.1em}1}}
\newcommand{\sigg}{\sigma}

\newcommand{\uu}{u}
\newcommand{\uut}{\overline{\uu}}
\newcommand{\uuu}{\uu'}
\newcommand{\uuuu}{\uu''}
\newcommand{\vv}{v}
\newcommand{\vvt}{\overline{\vv}}
\newcommand{\vvv}{v'}
\newcommand{\ww}{w}

\newcommand{\wwt}{\overline{w}}
\newcommand{\www}{\ww'}
\newcommand{\wwwt}{\overline{w'}}
\newcommand{\wwww}{\ww''}

\newcommand{\xx}{x}
\newcommand{\xxt}{\overline{x}}

\newcommand{\yy}{y}
\newcommand{\yyt}{\overline{y}}
\newintegeri{indi}{p}
\newintegeri{indii}{q}
\newintegeri{indiii}{r}
\newintegeri{indiv}{s}
\newintegeri{indv}{t}
\newintegeri{indvi}{t'}
\newintegeri{brdi}{b}
\newintegeri{brdii}{c}
\newintegeri{brdiii}{t}
\newintegeri{brdiv}{t'}
\newintegeri{ll}{\ell}
\newintegerii{d}
\newintegerii{e}
\newintegerii{h}
\newintegerii{i}
\newintegerii{j}
\newintegerii{k}
\newintegerii{m}
\newintegerii{n}
\newintegerii{p}
\newintegerii{q}
\newintegerii{r}
\newintegerii{s}
\newintegerii{t}

\title{A simple algorithm for finding short sigma-definite representatives}

\author{Jean Fromentin}
\author{Luis Paris}
\thanks{Both authors are partially supported by the \emph{Agence Nationale de la Recherche} (\emph{projet Th\'eorie de Garside}, ANR-08-BLAN-0269-03).}

\date{14 d\'ecembre 2010}
\theoremstyle{definition}
\newtheorem{algo}{Algorithm}
\newtheorem{defi}{Definition}[section]
\newtheorem{nota}[defi]{Notation}
\newtheorem{exam}[defi]{Example}
\theoremstyle{plain}
\newtheorem{lemm}[defi]{Lemma}
\newtheorem{prop}[defi]{Proposition}
\newtheorem{coro}[defi]{Corollary}

\renewcommand{\hh}{\noindent\hspace{1em}}
\newcommand{\hhh}{\noindent\hspace{2em}}

\begin{document}

\maketitle

\begin{abstract}
We describe a new algorithm which for each braid returns a quasi-geodesic $\sigma$-definite word representative, defined as a braid word in which the generator $\sigma_i$ with maximal index~$i$ appears either only positively or only negatively.

\end{abstract}

\section*{Introduction}
Since \cite{Dehornoy1994}, we know that Artin's braid groups $\BB\nn$ are left orderable, by an ordering that enjoys many remarkable properties.
This braid ordering is based on the property that every nontrivial braid admits a $\sigg$-definite representative, defined to be a braid word in standard Artin generators~$\sig\ii$ in which the generator~$\sig\ii$ with highest index $\ii$ appears either with only positive exponents or with only negative exponents.
In the past two decades, many different proofs of this result have been found, some of them based on algebra \cite{Burckel1997,Dehornoy1994,Dehornoy1997,Larue1994}, other on geometry~\cite{Bressaud2008,Dynnikov2007,Fenn1999}.
All these methods turn out to be algorithms. 
But in the best cases, starting with a braid word~$\ww$ of length~$\ll$, they only prove the existence of a $\sigg$-definite word $\www$ equivalent to~$\ww$ with length bounded by an exponential on $\ll$.
In \cite{Fromentin2008a}, an algorithm returning a quasi-geodesic $\sigg$-definite representative has been introduced.
It is heavily based on technical properties of the so-called rotating normal form on the Birman--Ko--Lee monoid.
Quite effective, this algorithm remains complicate.

The aim of this paper is to describe a simple algorithm returning a quasi-geodesic $\sigg$-definite representative.
It is based either on the alternating normal form introduced in~\cite{Dehornoy2008} or on the rotating normal form intoduced in \cite{Fromentin2008a,Fromentin2008,Fromentin2008b}.
The main advantage of this new algorithm is that it can be describe with few technical results on these normal forms. Part of the algorithm presented here uses some ideas from \cite{Fromentin2008a}. However, this new algorithm goes beyond the simplification of the previous one, and the paper can be read independently from~\cite{Fromentin2008a}.

The paper is organized as follows.
In Section~\ref{S:Rever}, we give an overview on reversing processes and give some elementary algorithm that will be needed to describe the main algorithm.
In Section~\ref{S:Splitting} we recall the definition of the $\fl\nn$-splitting, that is a natural way to describe each braid of~$\BP\nn$ from a finite sequence of braid of $\BP\nno$, and we give an algorithm to compute it.
In Section~\ref{S:Garside}, we introduce two different algorithms that allow us to express a braid of $\BB\nn$ as a quotient of braids lying in~$\BP\nn$.
In Section~\ref{S:Algo} we describe and prove the correctness of the main algorithm in the context of the alternating normal form.
Finally, in the last section, we investigate the complexity of our algorithm in the context of the Birman--Ko--Lee monoid~$\BKL\nn$.

\section{Reversing process}

\label{S:Rever}

In this section, we recall how to perform elementary computations in a finitely generated Garside monoid.
The main tool is the reversing algorithm introduced in~\cite{Dehornoy1999}.

Assume that $M$ is a Garside monoid.
Then, we define two partial orderings on $M$.
Given elements $\br$ and $\brr$ of $M$, we say that $\br$ \emph{left divides} (resp. \emph{right divides})~$\brr$, denoted by $\br\prec\brr$ (resp. $\br\succ\brr$), if there exists $\brbr$ in $M$ such that~$\br\,\brbr=\brr$ (resp.~$\br=\brbr\,\brr$) is satisfy. 

The \emph{left lcm} of two elements $\br$ and $\brr$ of $M$ is the minimal element~$\brbr$ in $M$, with respect to~$\prec$, satisfying $\br\prec\brbr$ and $\brr\prec\brbr$, and we denote it by~$\br\lcmL\brr$.
Of course, we define symmetrically the \emph{right lcm} of $\br$ and~$\brr$ in~$M$, which is denoted by~$\br\lcmR\brr$.

\begin{defi}
 Let $M$ be a monoid generated by a finite set $S$.

$(i)$ A word on the alphabet $S$ is called a \emph{positive $S$-word},

$(ii)$ A word on the alphabet $S\cup S\inv$ is called a \emph{$S$-word},

$(iii)$ The element represented by an $S$-word $\ww$ is denoted by $\wwt$,

$(iv)$ For $\ww,\www$ two $S$-words, we say that $\ww$ is equivalent to $\www$, denoted by $\ww\equiv\www$, if $\wwt=\wwwt$ holds.

\end{defi}

Let $M$ be a Garside monoid generated by a finite set $S$.
A \emph{left lcm selector} on~$S$ in $M$ is a mapping~$\compL:S\times S\to S^\ast$ such that, for all $\xx,\yy$ in $S$, the words~$\xx\,\compL(\xx,\yy)$ and $\yy\,\compL(\yy,\xx)$ both represent $\xxt\lcmL\yyt$.
We define also a \emph{right lcm selector} on $S$ in~$M$ to be a mapping $\compR$ such that $\compR(\xx,\yy)\,\yy$ and $\compR(\yy,\xx)\,\xx$ represent $\xxt\lcmR\yyt$ for all $\xx,\yy$ in $S$.

\begin{exam}
We recall that the positive braid monoid~$\BP\nn$ is defined for $\nn\ge2$ by the presentation
\begin{equation}
\label{E:BP:Presentation}
\left<\sig1,\Ldots,\sig\nno;\begin{array}{cl} \sig\ii\sig\jj\,=\,\sig\jj\sig\ii & \text{for $|\ii\minus\jj|\ge2$}\\ \sig\ii\sig\jj\sig\ii\,=\,\sig\jj\sig\ii\sig\jj & \text{for $|\ii\minus\jj|=1$}  \end{array}\right>^+.
\end{equation}
We put $\SBP\nn=\{\sig1,...,\sig\nno\}$. Then the applications $\compL$ and $\compR$ defined  on $\SBP\nn\times\SBP\nn$ by 
\[
\compL(\sig\ii,\sig\jj)=\compR(\sig\ii,\sig\jj)=\begin{cases}
                         \sig\jj&\text{for $\abs{i-j}\ge2$,}\\
			 \sig\jj\sig\ii&\text{for $\abs{i-j}=1$.}
                        \end{cases}
\]
are respectively left and right lcm selectors on $\SBP\nn$ in $\BP\nn$.
\end{exam}

For the rest of this section, we fix a Garside monoid~$M$, a finite generating set~$S$ of $M$, a left lcm selector $\compL$ and a right lcm selector~$\compR$ on $S$ in $M$.

\begin{defi}
Let~$\ww,\www$ be $S$-words. 
We say that $\ww\revR1\www$ is true if $\www$ is obtained from~$\ww$ by replacing a subword $\xx\inv\yy$ of $\ww$ by $\compL(\xx,\yy)\,\compL(\yy,\xx)\inv$.
We say that~$\ww\revvR\www$ is true if there exists a sequence $\ww=\ww_0,\Ldots,\ww_\kk=\www$ of $S$-words such that~$\ww_\ii\revR1\ww_\iip$ holds for all $\ii=0,1,\Ldots,\kk{-}1$.

Symmetrically, we say that $\ww\revvL\www$ is true, if $\www$ is obtained form $\ww$ by repeatedly replacing a subword $\xx\,\yy\inv$ of $\ww$ by the word $\compR(\yy,\xx)\inv\,\compR(\xx,\yy)$. 
\end{defi}

We now introduce the notion of  right reversing diagrams.
Assume that $\ww_0,\Ldots,\ww_\kk$ is a reversing sequence, \ie, a sequence of $S$-words such that $\ww_\ii\revR1\ww_\iip$ holds for each~$\ii$.
First we associate with~$\ww_0$ a path labelled with its successive letters: we associate to a positive letter~$\xx$ a horizontal right-oriented arrow labelled~$\xx$, and to a negative letter~$\xx\inv$ a vertical down-oriented arrow labelled~$\xx$.
Then we successively represent the $S$-words~$\ww_1,\Ldots,\ww_\kk$ as follows: if $\ww_\iip$ is obtained form $\ww_\ii$ by replacing the subword~$\xx\inv\yy$ of~$\ww_\ii$ by $\compL(\xx,\yy)\,\compL(\yy,\xx)\inv$, then we complete the pattern corresponding to~$\xx\inv\yy$ using a right-oriented arrow labelled~$\compL(\xx,\yy)$ and a down-oriented arrow labelled~$\compL(\yy,\xx)$ to obtain a square:

\begin{center}
\begin{picture}(82,22)
\put(2,2){\includegraphics[scale=0.7]{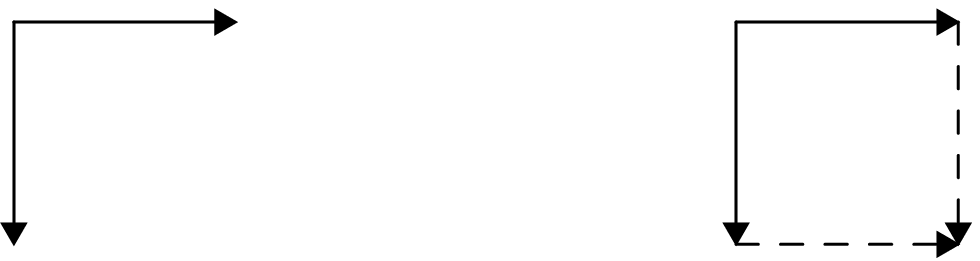}}
\footnotesize
\put(25,11){completed into}
\put(10,20){$\yy$}
\put(0.5,11){$\xx$}
\put(51.5,11){$\xx$}
\put(61,20){$\yy$}
\put(57,0){$\compL(\xx,\yy)$}
\put(71.5,11){$\compL(\yy,\xx)$}
\put(60,11){$\revvR$}
\end{picture}
\end{center}

Symmetrically, we define a left reversing diagram, in which we complete the pattern corresponding to~$\xx\yy\inv$ using a right-oriented arrow labelled~$\compR(\xx,\yy)$ and a down-oriented arrow labelled~$\compR(\yy,\xx)$:

\begin{center}
\begin{picture}(82,22)
\put(2,2){\includegraphics[scale=0.7]{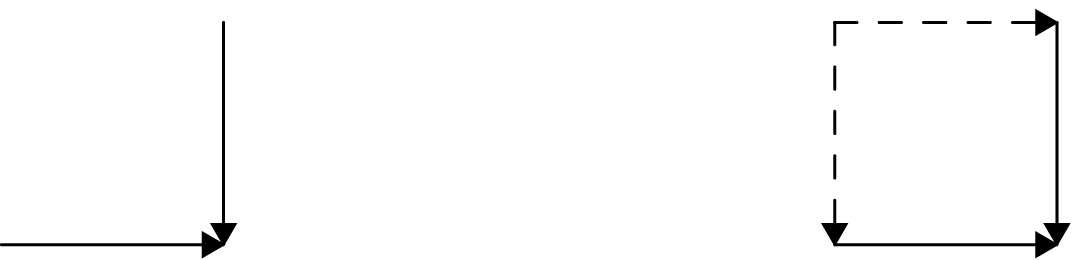}}
\footnotesize
\put(26,11){completed into}
\put(9,0.5){$\xx$}
\put(19,11){$\yy$}
\put(50,11){$\compR(\yy,\xx)$}
\put(64,20.5){$\compR(\xx,\yy)$}
\put(68.5,0.5){$\xx$}
\put(78.5,11){$\yy$}
\put(67.5,11){$\revvL$}
\end{picture}
\end{center}

\begin{prop}\cite{Dehornoy1999}
\label{P:Comp:Rev}
For every $S$-word $\ww$, there exist unique positive $S$-words $\uu$ and~$\vv$ such that $\ww\revvR\uu\,{\vv}\inv$ holds.
Moreover the words $\uu$ and $\vv$ are obtained from $\ww$ in time~$O(pos(\ww)\cdot neg(\ww))$, where $pos(\ww)$ is the number of positive letters occurring in $\ww$ and $neg(\ww)$ is the number of negative letters occurring in $\ww$.
A similar result occurs for $\revvL$.
\end{prop}

Let $\ww$ be a $S$-word.
As there exist unique positive $S$-words $\uu,\vv$ such that $\ww\revvR\uu\vv\inv$ holds, we say that $\uu$ is the \emph{right numerator} of~$\ww$, denoted by~$\numR(\ww)$, and that~$\vv$ is the \emph{right denominator} of~$\ww$, denoted by~$\denR(\ww)$.
Symmetrically, we define \emph{left numerator} and \emph{left denominator} of $\ww$ respectively denoted by~$\numL(\ww)$ and~$\denL(\ww)$. An immediate consequence of~\cite{Dehornoy1999}~is:

\begin{prop}
For all positive $S$-words $u,\vv$:

$(i)$ $\uut\prec\vvt$ holds if and only if $\denR(\uu\inv\vv)$ is the empty word $\varepsilon$,

$(ii)$ $\uut\succ\vvt$ holds if and only if $\denL(\uu\vv\inv)$ is the empty word $\varepsilon$.
\end{prop}

Let $\br,\brr$ be two elements of~$M$.
The \emph{left gcd} of $\br$ and $\brr$ is the maximal element~$\brbr$ with respect to $\prec$ such that $\brbr\prec\br$ and $\brbr\prec\brr$ holds.

\begin{prop}\cite[Proposition 7.7]{Dehornoy1999}
Let $\uu,\vv$ be positive $S$-words.
Then the left gcd of~$\uut$ and $\vvt$ is the element represented by
\begin{equation}
\label{E:WordLGcd}
\numL(\uu\denL(\numR(\uu\inv\vv)\denR(\uu\inv\vv)\inv)\inv).
\end{equation}
\end{prop}

See Figure~\ref{F:Gcd} for a description of \eqref{E:WordLGcd} in terms of reversing diagrams.

\begin{figure}[htbf!]
\label{F:Gcd}
\begin{picture}(70,38)
\put(12,2){\includegraphics[scale=0.7]{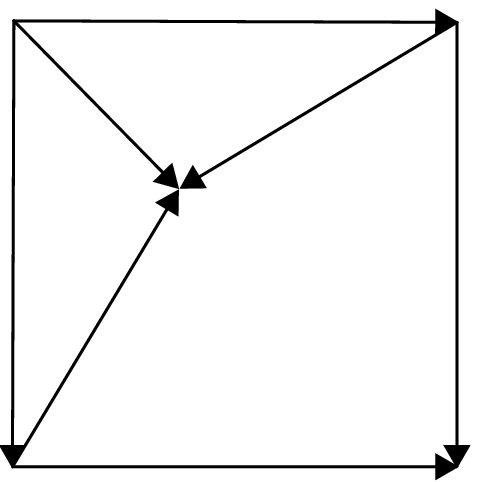}}
\footnotesize
\put(27,36){$\vv$}
\put(19,29.5){$\numL{(\uu{\uuuu}\inv)}$}
\put(16,20){\normalsize $\revvL$}
\put(29,22.5){$\numL(\uuu{\vvv}\inv)$}
\put(10,17){$\uu$}
\put(46,17){$\vvv{=}\denR(\uu\inv\vv)$}
\put(29,17){\normalsize $\revvL$}
\put(18,9){$\uuuu{=}\denL(\uuu{\vvv}\inv)$}
\put(19,0){$\uuu{=}\numR(\uu\inv\vv)$}
\end{picture}
\caption{\sf \smaller Reversing diagram corresponding to the computation of the left gcd of~$\uut$ and~$\vvt$. Firstly, we right reverse~$\uu\inv\vv$ to obtain~$\uuu{\vvv}\inv$. Secondly we left reverse~$\uuu{\vvv}\inv$ to compute~$\denL(\uuu{\vvv}\inv)$, denoted by~$\uuuu$.
Finally we left reverse~$\uu{\uuuu}\inv$ to compute~$\numL{(\uu{\uuuu}\inv)}$, which represents $\uut\gcdL\vvt$. }
\end{figure}

The reversing process takes a word on input and returns a word.
In order to simplify notations we shall use divisor and gcd symbols on words.

\begin{nota}
Let $\uu,\vv$ be positive $S$-words.

$(i)$ If $\uut\prec\vvt$ holds, we denote by $\uu\backslash\vv$ the word $\numR(\uu\inv\,\vv)$.

$(ii)$ If $\uut\succ\vvt$ holds, we denote by $\uu/\vv$ the word $\numL(\uu\,\vv\inv)$.

$(iii)$ We denote by $\uu\gcdL\vv$ the word of \eqref{E:WordLGcd}.
\end{nota}

With these notations, the element $\overline{\uu\backslash\vv}$ is equal to $\uut\inv\vvt$, the element $\overline{\uu/\vv}$ is equal to~$\uut\,\vvt\inv$ and the element $\overline{\uu\gcdL\vv}$ is the left gcd of $\uut$ and $\vvt$.

In the sequel we will consider two Garside monoids naturally generated by a finite set, namely the positive braid monoid~$\BP\nn$ generated by $\SBP\nn$ and the dual braid monoid $\BKL\nn$ generated by~$\SBKL\nn$ (see Section~\ref{S:Dual}).
From now on, we will not specify the lcm selectors for left and right reversing operations in these monoids, if not needed.

\section{The $\fl\nn$-splitting}
\label{S:Splitting}
It is shown in \cite{Dehornoy2008} how associate with every braid $\br$ of $\BP\nn$ a unique sequence of braids in $\BP\nno$, called the \emph{$\fl\nn$-splitting} of $\br$, that completely determines $\br$.
As mentioned in the introduction, our algorithm is based on this operation.
In this section we recall the definition and the construction of the $\fl\nn$-splitting of a braid.

We recall that the positive braid monoid~$\BP\nn$ is a Garside monoid with Garside element $\Delta_\nn$ defined by
\[
\Delta_\nn=(\sig1\,\Ldots\,\sig\nno)\cdot(\sig1\,\Ldots\,\sig\nnt)\cdot\Ldots\cdot(\sig1\,\sig2)\cdot\sig1.
\]
See \cite{Dehornoy1999,Garside1969} for a definition of a Garside monoid.

We denote by $\fl\nn$ the \emph{flip automorphism} of $\BP\nn$, \ie, the application defined on~$\BP\nn$ by $\fl\nn(\br)=\Delta_\nn\,\br\,\Delta_\nn\inv\,$.
The initial observation of the construction of the alternating normal form is that each braid of $\BP\nn$ admits a unique maximal right divisor lying in $\fl\nn^\kk(\BP\nno)$ for all $\kk$.

\begin{lemm}
\label{L:Tail}
 For $\nn\ge3$ and $\kk\ge0$, every braid~$\br$ of~$\BP\nn$ admits a unique maximal right divisor~$\br_1$ lying in~$\fl\nn^\kk(\BP\nno)$.
\end{lemm}

\begin{proof}
The braid $\br_1$ is a maximal right divisor of~$\br$ lying in $\fl\nn^\kk(\BP\nno)$ if and only if $\fl\nn^{-\kk}(\br_1)$ is the maximal right divisor of $\fl\nn^{-\kk}(\br)$ lying in $\BP\nno$.
As the submonoid~$\BP\nno$ of~$\BP\nn$ is closed under right divisors and left lcm, the braid  $\fl\nn^{-\kk}(\br_1)$ exists and is unique.
\end{proof}

\begin{defi}
 The braid $\br_1$ of Lemma~\ref{L:Tail} is called the $\fl\nn^\kk(\BP\nno)$-tail of $\br$. 
\end{defi}

By iterating the tail construction, we then associate with every braid of $\BP\nn$ a finite sequence of braids of $\BP\nno$.

\begin{prop}\cite[Proposition 2.5]{Dehornoy2008}\label{P:Splitting}
 Assume $\nn\ge3$. Then for each nontrivial braid $\br$ of $\BP\nn$, there exists a unique sequence $(\br_\brdi,\Ldots,\br_1)$ in $\BP\nno$ satisfying 
\begin{gather}
 \br_\brdi\not=1 \quad \textrm{and} \quad \br=\fl\nn^\brdio(\br_\brdi)\cdot...\cdot\fl\nn(\br_2)\cdot\br_1\\
\textrm{for each $\kk\ge1$,  $\fl\nn^\kko(\br_\kk)$ is the $\fl\nn^\kko(\BP\nno)$-tail of $\fl\nn^\brdio(\br_\brdi)\cdot...\cdot\fl\nn^\kko(\br_\kk)$}
\end{gather}

\begin{defi}
\label{D:Splitting}
The sequence $(\br_\brdi,\Ldots,\br_1)$ of Proposition~\ref{P:Splitting} is called the $\fl\nn$-splitting of~$\br$ and its length is called the $\fl\nn$-breadth of $\br$.

\end{defi}
\end{prop}

We give now an algorithm to compute the $\fl\nn$-splitting of a braid given by a positive~$\SBP\nn$-word~$\ww$.
More precisely the algorithm returns a sequence $(\ww_\brdi,\Ldots,\ww_1)$ of positive $\SBP\nno$-words such that $(\wwt_\brdi,\Ldots,\wwt_1)$ is the $\fl\nn$-splitting of~$\wwt$.

\begin{algo}\label{A:PhiNSplitting}
Compute the $\fl\nn$-splitting of the braid represented by $\ww$ 

\hh\textbf{Input:} A positive $\SBP\nn$-word $\ww$ with $\nn\ge3$

\hhh \verb|1.| \verb|Put| $\ss=(\,)$, $\www=\ww$ and $\kk=0$.

\hhh \verb|2.| \verb|While| $\www\not=\varepsilon$ \verb|do|

\hhh \verb|3.| \hh \verb|Put| $\uu=\varepsilon$.

\hhh \verb|4.| \hh \verb|While| there exists $\xx\in \SBP\nno$ such that $\wwwt\succ\fl\nn^\kk(\xxt)$ \verb|do| 

\hhh \verb|5.| \hh \hh \verb|Put| $\www=\www/\fl\nn^\kk(\xx)$ and $\uu=\xx\,\uu$

\hhh \verb|6.| \hh Insert $\uu$ on the left of $\ss$.

\hhh \verb|7.| \hh \verb|Put| $\kk=\kk+1$

\hhh \verb|8.| \verb|Return| $s$.
\end{algo}

\begin{prop}
Running on $\ww$, Algorithm~\ref{A:PhiNSplitting} ends in time $O(\len{\ww}^2)$ and returns a sequence $(\ww_\brdi,\Ldots,\ww_1)$ of positive $\SBP\nno$-words such that  $(\wwt_\brdi,\Ldots,\wwt_1)$ is the $\fl\nn$-splitting of $\wwt$.
\end{prop}

\begin{proof}
We denote by $\br$ the braid represented by the value of $\www$ at Line \verb|3|. 
Lines~\verb|3|, \verb|4| and \verb|5| compute the maximal right divisor of $\br$ lying in~$\fl\nn^\kk(\BP\nno)$.
At Line \verb|6|, the braid $\fl\nn^\kk(\uut)$ is equal to the $\fl\nn^\kk(\BP\nno)$-tail of $\br$ and $\wwwt$ is equal to $\br/\fl\nn^\kk(\uut)$.

Therefore the algorithm applies successively the $\fl\nn^\kk(\BP\nno)$-tail construction for $\kk=1,2,\Ldots$.
Then, by Proposition~\ref{P:Splitting} it must stop and return the expected sequence of words.

As for time complexity, testing if $\wwwt\succ\fl\nn^\kk(\xxt)$ holds and computing $\www/ \fl\nn^\kk(\xx)$ need to run the left revering process on $\www\,(\fl\nn^\kk(\xx))\inv$.
Proposition \ref{P:Comp:Rev} guarantees that these two operations can be done in time $O(\len{\www})$, and so, in time $O(\len{\ww})$ since~$\len{\www}\le\len{\ww}$ holds.
Then an easy bookkeeping shows that the algorithm ends in time~$O(\len{\ww}^2)$.
\end{proof}

\section{Garside quotient}

\label{S:Garside}

In the previous section we have seen how to compute the $\fl\nn$-splitting of a braid lying in $\BP\nn$.
Of course there is no possible extension of the notion of $\fl\nn$-splitting to the braid group~$\BB\nn$.
However, we have the following.

\begin{prop}
\label{P:GarsideQuotient}
 Each braid $\br$ admits a unique decomposition $\Delta_\nn^{-\tt}\,\brr$ where $\tt$ is a nonnegative integer and $\brr$ is a braid belonging to $\BP\nn$, which is not left divisible by~$\Delta_\nn$, unless $\tt=0$.

\end{prop}

\begin{proof}
The monoid $\BP\nn$ is a Garside monoid with Garside elements $\Delta_\nn$, see \cite{Garside1969}.
 As $\BB\nn$ is the group of fractions of $\BP\nn$, there exist a smallest nonnegetive integer $\tt$ such that $\Delta_\nn^\tt\,\br$ lies in $\BP\nn$.
  If $\tt$ is positive, the minimality hypothesis on $\tt$ implies $\Delta_\nn\not\prec\Delta_\nn^\tt\,\br$. Then we define $\brr$ to be the braid $\Delta_\nn^\tt\,\br$.

Assume now that $\Delta_\nn^{-\tt'}\,\brrr$ is another decomposition of $\br$ satisfying the hypothesis of the proposition.
As $\Delta_\nn^{\tt'}\br$ belongs to $\BP\nn$, we have $\tt'\ge\tt$.
Assume $\tt'>\tt$.
Then we have $\tt'>0$.
As the braid~$\brrr$ is equal to $\Delta_\nn^{\tt'-\tt}\,\brr$, the relation $\Delta_\nn\prec\brrr$ holds, that is in contradiction with  $\tt'>0$ and the hypothesis of the proposition.
Hence $\tt'$ is equal to $\tt$ and then $\brr$ is equal to $\brrr$.
\end{proof}

\begin{algo}\label{A:GarsideFraction}
Compute the decomposition $\Delta_\nn^{-\tt}\,\vvt$ given in Proposition \ref{P:GarsideQuotient} of the braid represented by the $\SBP\nn$-word $\ww$.

\hh\textbf{Input:} An $\SBP\nn$-word $\ww$

\hhh \verb|1.| Write $\ww$ as $\ww_0\,\xx_1\inv\,\ww_1\,\Ldots\,\ww_\tto\,\xx_\tt\inv\,\ww_\tt$ (where $w_i$ is a positive word and $\xx_\jj$ is a letter).

\hhh \verb|3.| \verb|For| $\ii=1\,\Ldots\,\tt$ compute $\uu_\ii$ such that $\Delta_\nn=\uu_\ii\,\xx_\ii$.

\hhh \verb|4.| \verb|Put| $\vv=\fl\nn^\tt(\ww_0)\,\fl\nn^\tto(\uu_1\,\ww_1)\,\Ldots\,\fl\nn(\uu_\tto\,\ww_\tto)\,\uu_\tt\,\ww_\tt$.

\hhh \verb|5.| \verb|While| $\Delta_\nn\prec\vvt$ and $\tt>0$ hold \verb|do|

\hhh \verb|6.| \hh \verb|Put| $\vv=\Delta_\nn\backslash\vv$ and $\tt=\tt-1$.

\hhh \verb|7.| \verb|Return| $\Delta_\nn^{-\tt}\,\vv$.
\end{algo}

\begin{prop}
\label{P:A:GarsideFraction}
Running on $\ww$, Algorithm~\ref{A:GarsideFraction} ends in time $O(\len\ww^2)$ and has the correct output.
Moreover we have $\len{\Delta_\nn^{-\tt}\uu}\le(\nn^2{-}\nn{-}1)\cdot\lens{\wwt}$, where $\lens{\br}$ is the minimal length of a $\Sigma$-word representing $\br$.
\end{prop}

\begin{proof}
For  $\ii=1,\Ldots,\tt$, we denote by $\xx_\ii\inv$ the negative letters occurring in $\ww$.
Then we replace each $\xx_\ii\inv$ by $\Delta_\nn\inv\,\uu_\ii$ to obtain 
\begin{equation}
\label{E:P:A:GarsideFraction}
\ww\equiv\ww_0\,\Delta_\nn\inv\,\uu_1\,\ww_1\,\Ldots\,\ww_\tto\,\Delta_\nn\inv\,\uu_\tt\,\ww_\tt.
\end{equation}
The definition of $\fl\nn$ implies $\uu\,\Delta_\nn\inv\equiv\Delta_\nn\inv\,\fl\nn(\uu)$ for every positive $\SBP\nn$-word $\uu$.
From relation \eqref{E:P:A:GarsideFraction}, we obtain 
\begin{equation}
\ww\equiv\Delta_\nn^{-\tt}\fl\nn^\tt(\ww_0)\,\fl\nn^\tto(\uu_1\,\ww_1)\,\Ldots\,\fl\nn(\uu_\tto\,\ww_\tto)\,\uu_\tt\,\ww_\tt.
\end{equation}
So the word $\vv$ introduced in Line \verb|4| is equivalent to $\Delta_\nn^\tt\,\ww$.
After Lines \verb|5| and \verb|6|, the braid $\vv$ is not left divisible by $\Delta_\nn$ unless $\tt=0$.
At the end, we have $\ww\equiv\Delta_\nn^{-\tt}\,\vv$, hence the algorithm returns the correct output.

As for the length, replacing $\xx_\ii\inv$ by $\Delta_\nn\inv\,\uu_\ii$ multiplies it by at most $2\len{\Delta_\nn}-1$, \ie, by at most~$\nn^2{-}\nn{-}1$.
Indeed, the relations in the presentation~\eqref{E:BP:Presentation} preserve the length, hence we have~$\len{\uu_\ii}=\len{\Delta_\nn}{-}1$.
By Proposition~\ref{P:GarsideQuotient}, the integer $\tt$ and the braid~$\vvt$ depend only of the braid~$\wwt$ and not on the word~$\ww$.
Hence, applying the algorithm to a geodesic word representing~$\wwt$ gives 
\[
\len{\Delta_\nn^{-\tt}\vv}\le(\nn^2{-}\nn{-}1)\cdot\lens{\wwt}.
\]

As for time complexity, the word of Line \verb|4| is obtained in time $O(\len{\ww})$.
The \verb|while| command of Line \verb|5| needs at most $\len\vv$ steps.
Testing if $\Delta_\nn\prec\vvt$ holds and computing $\Delta_\nn\backslash\vv$ need to run the right reversing process on $\Delta_\nn\inv\,\vv$.
Proposition~\ref{P:Comp:Rev} guarantees that these two operations can be done in time $O(\len{\vv})$. 
Then, from $\len{\vv}\le(\nn^2{-}\nn{-}1)\cdot\len{\ww}$, we deduce that the algorithm ends in time $O(\len{\ww})^2$.
\end{proof}

Now, for a braid $\br$, the decomposition which we shall introduce in the next proposition is called the Garside--Thurston normal form of $\br$. 
We use this normal form for computing the minimal $\kk$ such that $\br$ lies in $\BB\kk$.
Indeed, as we will see, the Garside-Thurston normal form depends only on $\br$ and not on the group $\BB\nn$ in which it is viewed.

\begin{prop}\cite[Corollary 7.5]{Dehornoy1999}
\label{P:GarsideThurston}
 Each braid $\br$ of $\BB\nn$ admits a unique decomposition~${\brr}\inv\,\brrr$ where~$\brr,\brrr$ belong to $\BP\nn$ and such that $\brr\gcdL\brrr$ is trivial.
 Moreover if $\br$ is represented by $\ww$ then the braid $\brr$ is represented by $\denL(\numR(\ww)\denR(\ww)\inv)$ and the braid $\brrr$ is represented by~$\numL(\numR(\ww)\denR(\ww)\inv)$.
\end{prop}

Since $for$ $\kk\le\nn$ the lattice operation $\gcdL$ in $\BP\kk$ coincides with that of $\BP\nn$, a direct consequence of Proposition~\ref{P:GarsideThurston} is the following.

\begin{coro}
\label{C:GarsideThurston}
Let $\kk\le\nn$,  $\br$ in $\BB\nn$ and $\br={\brr}\inv\brrr$ be the Garside-Thurston normal form of~$\br$. We have $\br\in\BB\kk$ if and only if $\brr,\brrr$ lie in $\BP\kk$.
\end{coro}

\begin{defi}
 We define the \emph{index} of a $\SBP\nn$-word $\ww$ to be the maximal $\ii$ such that~$\ww$ contains a letter $\sig\iio$. The \emph{index} of a braid $\br$ is the minimal index of a word which represents $\br$.
\end{defi}

Obviously, the index of a braid~$\br$ is the minimal integer $\nn$ such that $\br$ lies in $\BB\nn$.

\begin{algo}\label{A:GarsideThurston}
Compute the index $\kk$ of $\wwt$ and a $\SBP\kk$-word $\wwww$ equivalent to $\ww$.

\hh\textbf{Input:} An $\SBP\nn$-word $\ww$

\hhh \verb|1.| Right reverse $\ww$ into $\www$.

\hhh \verb|2.| Left reverse $\www$ into $\wwww$.

\hhh \verb|3.| Let $\kk$ be the index of $\wwww$.

\hhh \verb|4.| \verb|Return| $(\kk,\wwww)$. 
\end{algo}

The correctness of this algorithm is a direct consequence of Proposition~\ref{P:GarsideThurston} together with Corollary~\ref{C:GarsideThurston}.
Moreover, by Proposition~\ref{P:Comp:Rev} it ends in time~ $O(\len{\ww}^2)$.

\section{The main algorithm}

\label{S:Algo}
Putting all pieces together, we can now describe our algorithm which returns a quasi-geodesic word equivalent to a given word.
However, we first recall the definition of $\sigg$-definite words and give the result of~\cite{Dehornoy2008} which will be used to prove the correctness of the algorithm.
As ever we assume $\SBP\nn\subset\SBP\nnp$ (as well as $\BB\nn\subset\BB\nnp)$ for all $n\ge2$, and we set $\Sigma=\bigcup_{\nn=2}^{\infty}\SBP\nn$.

\begin{defi}~

$(i)$ A $\Sigma$-word is said to be $\sig\ii$-positive (resp. $\sig\ii$-negative) if it contains at least one letter of the form $\sig\ii$, no letter $\siginv\ii$ (resp. at least one letter~$\siginv\ii$, no letter $\sig\ii$) and no letter $\sig\jj$ with $\jj>\ii$.

$(ii)$ A $\Sigma$-word is said to be $\sigg$-definite if it is either trivial, or $\sig\ii$-positive or $\sig\ii$-negative for a certain $\ii$.

$(iii)$ A braid is said to be $\sig\ii$-positive (resp. $\sig\ii$-negative) if it can be represented by a $\sig\ii$-positive word (resp. a $\sig\ii$-negative word).
\end{defi}

Recall form \cite{Dehornoy1994} that the celebrated Dehornoy ordering on $\BB\nn$ is defined by $\br<\brbr$ if $\br\inv\brbr$ is 
$\sig\ii$-positive for some $i\le\nn$.
The key property that will be used on the $\fl\nn$-splitting operation is its coincidence with the Dehornoy ordering $<$.

\begin{prop}\cite{Dehornoy2008}
\label{P:Coincidence}
 Let $\br$ and $\brbr$ be two braids of $\BP\nn$.
Let $(\br_\brdi,\Ldots,\br_1)$ and $(\brbr_{\brdii},\Ldots,\brbr_1)$ be the $\fl\nn$-splittings of $\br$ and $\brbr$ respectively.
Then $\br<\brbr$ holds  if and only if we have either $\brdi<\brdii$ or $\brdi=\brdii$ and, for some $\brdiii\le\brdi$, we have $\br_{\brdiv}=\brbr_{\brdiv}$ for $\brdiii<\brdiv\le\brdi$ together with $\br_\brdiii<\brbr_\brdiii$ .
\end{prop}

In \cite{Dehornoy2008}, Dehornoy proves that the minimal positive braid of a given  $\fl\nn$-breadth~$\brdipp$ (see \ref{D:Splitting}) with $\brdi\ge0$ is $\widehat\Delta_{\nno,\brdi}=\Delta_\nn^{\brdi}\,\Delta_\nno^{-\brdi}$, and the $\fl\nn$-splitting of the latter is the following sequence of length $\brdipp$
\begin{equation}
(\sig1, \sig\nno\,...\,\sig2\,\sig1^2,\Ldots,\sig\nno\,...\,\sig2\,\sig1^2,\sig\nno\,...\,\sig2\,\sig1,1).
\end{equation}
As the braid $\Delta_\nno$ lies in $\BP\nno$, we deduce that the $\fl\nn$-splitting of $\Delta_\nn^\brdi$ is the following sequence of length $\brdipp$
\begin{equation}
 (\sig1, \sig\nno\,...\,\sig2\,\sig1^2,\Ldots,\sig\nno\,...\,\sig2\,\sig1^2,\sig\nno\,...\,\sig2\,\sig1,\Delta_\nno^{\brdio}).
\end{equation}

The idea of our algorithm is that we can easily decide if a quotient $\Delta_\nn^{-\tt}\,\br$ with $\br$ lying in $\BP\nn$ is $\sig\nno$-negative or not.

\begin{lemm}
\label{L:Coincidence}
Assume that $\br$ is a braid of $\BP\nn$ such that $\widehat\Delta_{\nno,\brdi}\le\br\le\Delta_\nn^\brdi$ holds, then the quotient~$\Delta_\nn^{-\brdi}\br$ lies in $\BB\nno$.
\end{lemm}

\begin{proof}
 The relation $\widehat\Delta_{\nno,\brdi}\le\br\le\Delta_\nn^\brdi$ and Proposition~\ref{P:Coincidence} imply that the $\fl\nn$-splitting of $\br$ is the following sequence of length $\brdipp$
\begin{equation}
  ( \sig1, \sig\nno\,...\,\sig2\,\sig1^2,\Ldots,\sig\nno\,...\,\sig2\,\sig1^2,\sig\nno\,...\,\sig2\,\sig1,\br_1),
\end{equation}
with $1\le\br_1\le\Delta_\nno^{\brdi}$.
Hence $\br$ is equal to $\widehat\Delta_{\nno,\brdi}\,\br_1$ where $\br_1$ belongs to $\BP\nno$.
Then, as the quotient $\Delta_\nn^{-\brdi}\,\br$ is equal to $\Delta_\nno^{-\brdi}\,{\widehat\Delta_{\nno,\brdi}}\,\br$, we obtain $\Delta_\nn^{-\brdi}\,\br=\Delta_\nno^{-\brdi}\,\br_1$.
As $\br_1$ and $\Delta_\nno$ lie in~$\BB\nno$ the braid $\Delta_\nn^{-\brdi}\,\br$ lies in $\BB\nno$.
\end{proof}

\begin{prop}
\label{P:Quotient}
Assume $\nn\ge3$ and $\br$ is a braid of $\BP\nn$.
Let $\tt$ be a positive integer  and $\brdi$ the~$\fl\nn$-breadth of $\br$.
If $\tt\ge\brdio$ holds then the quotient $\Delta_\nn^{-\tt}\,\br$ is $\sig\nno$-negative.
Otherwise it is not $\sig\nno$-negative.
\end{prop}

\begin{proof}
 Let $(\br_\brdi,\Ldots,\br_1)$ be the $\fl\nn$-splitting of $\br$.
Then the braid $\Delta_\nn^{-\tt}\br$ is equal to
\begin{equation}
\label{E:P:Quotient:1}
\Delta_\nn^{-\tt}\cdot\fl\nn^\brdio(\br_\brdi)\cdot\Ldots\cdot\fl\nn(\br_2)\cdot\br_1.
\end{equation}
Pushing $\brdio$ powers of $\Delta_\nn$ to the right in \eqref{E:P:Quotient:1} and dispatching them between the factors $\br_\kk$, we find
\begin{align*}
 \Delta_\nn^{-\tt}\br&\equiv\Delta_\nn^{-\tt}\cdot\fl\nn^\brdio(\br_\brdi)\cdot\Ldots\cdot\fl\nn(\br_2)\cdot\br_1\\
&\equiv\Delta_\nn^{-\tt+\brdio}\cdot\Delta_\nn^{-\brdio}\fl\nn^\brdio(\br_\brdi)\cdot\Ldots\cdot\fl\nn(\br_2)\cdot\br_1\\
&\equiv\Delta_\nn^{-\tt+\brdio}\cdot\br_\brdi\cdot\Delta_\nn\inv\cdot\Delta_\nn^{-\brdit}\cdot\Ldots\cdot\fl\nn(\br_2)\cdot\br_1\\
&\equiv ... \equiv \Delta_\nn^{-\tt+\brdi-1}\ \br_\brdi\,\Delta_\nn\inv\ \br_\brdio\,\Delta_\nn\inv\,\Ldots\,\br_2\,\Delta_\nn\inv\ \br_1.
\end{align*}
If the relation $\tt\ge\brdio$ holds then the braid
\begin{equation}
\label{E:P:Quotient:2}
\Delta_\nn^{-\tt+\brdi-1}\ \br_\brdi\,\Delta_\nn\inv\ \br_\brdio\,\Delta_\nn\inv\,\Ldots\,\br_2\,\Delta_\nn\inv\ \br_1, 
\end{equation}
 is $\sig\nno$-negative.
Indeed, by definition, the braid $\Delta_\nn\inv$ is $\sig\nno$-negative, while for each $\kk$ the braid~$\br_\kk$ lies in $\BP\nno$.
So, as $-\tt+\brdi-1$ is nonpositive, the expression \eqref{E:P:Quotient:2} contains $\tt$ letters $\siginv\nno$ and no letter $\sig\nno$.

Now assume $\tt<\brdit$.
The $\fl\nn$-breadth of $\Delta_\nn^\tt$ is $\ttpp$ and we have $\ttpp<\brdi$.
Then Proposition~\ref{P:Coincidence} implies $\Delta_\nn^\tt<\br$, \ie, $\Delta_\nn^{-\tt}\,\br$ is $\sig\ii$-positive for a certain $\ii$, hence it is not $\sig\nno$-negative.

Finally assume $\tt=\brdit$.
If the relation $\Delta_\nn^\tt<\br$ holds we concluse as in the previous case.
Then assume $\br\le\Delta_\nn^\tt$.
As the $\fl\nn$-breadth of $\br$ is $\brdi$, which is equal to~$\ttpp$, Proposition~\ref{P:Coincidence} implies $\widehat\Delta_{\nno,\tt}\le\br$.
Then by Lemma~\ref{L:Coincidence} the quotient ~$\Delta_\nn^{-\tt}\,\br$ lies in $\BB\nno$, hence it is not $\sig\nno$-negative.
\end{proof}

The following algorithm takes in entry a braid word $w$ representing a braid~$\br$ and it returns a $\sigma$-definite word representing $\br$. The main idea is to bring all possible cases to the case where~$\br$ is $\sig\nno$-negative, \ie, when $\br$ satisfy the conditions of~Proposition~\ref{P:Quotient}.

\begin{algo}\label{A:Main}
Compute a $\sigma$-definite representative.

\hh\textbf{Input:} A $\SBP\nn$-word $\ww$

\hhh \verb|1.| \verb|Put| $\ee=1$.

\hhh \verb|2.| Let $(\kk,\vv)$ be the output of Algorithm~\ref{A:GarsideThurston} applied to $\ww$.

\hhh \verb|3.| Use Algorithm~\ref{A:GarsideFraction} to compute $\Delta_\kk^{-\tt}\,\uu\equiv\vv^\ee$.

\hhh \verb|4.| \verb|If| $\tt=0$ or $\kk=2$ \verb|then return| $\Delta_\kk^{-\tt}\,\uu$.

\hhh \verb|5.| Use Algorithm~\ref{A:PhiNSplitting} to compute the $\fl\kk$-splitting $(\uu_\brdi,\Ldots,\uu_1)$ of $\uu$.

\hhh \verb|6.| \verb|If| $\tt\ge\brdio$ \verb|then return| $(\Delta_\kk^{-\tt+\brdi-1}\ \uu_\brdi\,\Delta_\kk\inv\ \uu_\brdio\,\Delta_\kk\inv\,\Ldots\,\uu_2\,\Delta_\kk\inv\ \uu_1)^\ee$.

\hhh \verb|7.| \verb|Else put| $\ee=-1$ and \verb|goto Line 3|.

\end{algo}

\begin{prop}
\label{P:Main}
 Algorithm~\ref{A:Main} ends and returns in time $O(\len\ww)^2$ a $\sigg$-definite word~$\www$ equivalent to $\ww$ with $\len{\www}\le(\nn^2-\nn-1) \cdot\lens\wwt$, where $\lens{\br}$ is the minimal length of a $\Sigma$-word representing~$\br$.
\end{prop}

\begin{proof}
We use Algorithm~\ref{A:GarsideThurston} to compute the index $\kk$ of $\wwt$ and a 
 $\SBP\kk$-word $\vv$ equivalent to $\ww$.
In particular, the braid $\wwt$ is either $\sig\kko$-positive or $\sig\kko$-negative.

Next, we use Algorithm~\ref{A:GarsideFraction} to compute a quotient $\Delta_\kk^{-\tt}\,\uu$ that is equivalent to $\vv$ and so to $\ww$.
By Proposition~\ref{P:GarsideQuotient} the exponent $\tt$ and the positive braid $\uut$ depends only of the braid $\wwt$ and not on the word $\vv$.
Moreover, we have 
\[
\len{\Delta_\kk^{-\tt}\uu}\le(\kk^2{-}\kk{-}1)\lens{\wwt}\le(\nn^2{-}\nn{-}1)\lens{\wwt}.
\]
If $\tt$ is equal to $0$ then the quotient $\Delta_\kk^{-\tt}\uu$ is equal to $\uu$, that is a positive word, hence a $\sigg$-definite word.
If $\nn$ is equal to $2$ with $\tt\not=0$ then $\uu$ is empty and $\Delta_\kk^{-\tt}\uu$ is equal to~$\Delta_\kk^{-\tt}$, that is a negative word, hence a $\sigg$-definite word.

Next, we use Algorithm~\ref{A:PhiNSplitting} to compute the $\fl\kk$-splitting $(\uut_\brdi,\Ldots,\uut_1)$ of $\uut$.
Then the word $\Delta_\kk^{-\tt}\uu$ is equivalent to $\uuu$ defined by
\begin{equation}
\label{E:P:Main:2}
\uuu=\Delta_\kk^{-\tt+\brdi-1}\ \uu_\brdi\,\Delta_\kk\inv\ \uu_\brdio\,\Delta_\kk\inv\,\Ldots\,\uu_2\,\Delta_\kk\inv\ \uu_1. 
\end{equation}

If the relation $\tt\ge\brdio$ holds then the word $\uuu$ is $\sig\kko$-negative---see proof of~Proposition~\ref{P:Quotient}---hence it is $\sigg$-definite.
So in this case the algorithm returns a $\sigg$-definite word equivalent to~$\ww$.

Now, assume $\tt <\brdio$.
In this case we redo the same process with the word $\ww\inv$.
Note that, as the index of $\ww$ and $\ww\inv$ are the same, we can directly go to Line \verb|3| of the algorithm.
In this case, by Proposition~\ref{P:Quotient}, the braid $\wwt$ is not $\sig\kko$-negative, \ie, it is $\sig\kko$-positive or it lies in $\BB\mm$ with $\mm<\kk$.
As $\kk$ is the index of $\wwt$, the braid $\wwt$ is $\sig\kko$-positive.
So the braid represented by~$\vv\inv$ is $\sig\kko$-negative.
Hence the new value of $\tt$ and $\brdi$ satisfy the relation~$\tt\ge\brdio$ and the algorithm ends.

For length complexity, the length of the $\SBP\kk$-word $\uuu$ given in  \eqref{E:P:Main:2} is equal to the length of the $\SBP\kk$-word $\Delta_\kk^{-\tt}\,\uu$.
By Proposition~\ref{P:GarsideQuotient} we have  
\[
\Delta_\kk^{-\tt}\,\uu\le(\kk^2{-}\kk{-}1)\cdot\lens{\vvt^\ee}=(\kk^2{-}\kk{-}1)\cdot\lens{\wwt^\ee}\le(\nn^2{-}\nn{-}1)\cdot\lens{\wwt^\ee}.
\]
Then, $\lens{\wwt\inv}=\lens{\wwt}$ implies $\len{\uuu}\le(\nn^2{-}\nn{-}1)\lens{\wwt}$.
\end{proof}

\section{Dual braid monoid}

\label{S:Dual}

The \emph{dual braid monoid} is another submonoid of $\BB\nn$. It is generated by a subset of $\BB\nn$ that properly contains $\{\sig1,\ldots,\sig\nno\}$, and consists of the so-called \emph{Birman-Ko-Lee generators} introduced in~\cite{Birman1998}.

\begin{defi}\label{D:Aij}
~

$(i)$  For $1\le\indi<\indii$, we put $\aa\indi\indii=\sig\indi\Ldots\sig\indiit\ \sig\indiio\ \siginv\indiit\Ldots\siginv\indi$.

$(ii)$ For $\nn\ge2$, the set $\SBKL\nn$ is defined to be $\{\aa\indi\indii\mid1\le\indi<\indii\le\nn\}$.

$(iii)$ The \emph{dual braid monoid}~$\BKL\nn$, is the submonoid of $\BB\nn$ generated by $\SBKL\nn$.
\end{defi}

For $\indi<\indii$, we denote by $\llbracket\indi,\indii\rrbracket$ the interval $\{\indi,\Ldots,\indii\}$ of $\mathbb{N}$, and we say that $\llbracket\indi,\indii\rrbracket$ is nested in~$\llbracket\indiii,\indiv\rrbracket$ if we have $\indiii<\indi<\indii<\indiv$.
A presentation of $\BKL\nn$ in terms of $\aa\indi\indi$ is as follows.

\begin{prop}\cite{Birman1998}
\label{P:BKL} In terms of the $\aa\indi\indii$, the monoid $\BKL\nn$ is presented by
\begin{align*}
 \aa\indi\indii\,\aa\indiii\indiv&=\aa\indiii\indiv\,\aa\indi\indii \quad \text{for $\llbracket\indi,\indii\rrbracket$ and $\llbracket\indiii,\indiv\rrbracket$ disjoint or nested},\\
\aa\indi\indii\,\aa\indii\indiii&= \aa\indii\indiii\, \aa\indi\indiii = \aa\indi\indiii\, \aa\indi\indii \quad \text{for $1\le\indi<\indii<\indiii\le\nn$.}
\end{align*}
\end{prop}
As the positive braid monoid, we can endow the Birman--Ko--Lee monoid~$\BKL\nn$ with a Garside structure. The corresponding Garside element is 
\[
\ddd\nn=\aa12\,\aa23\,\ldots\,\aa\nno\nn.
\] 

We denote by $\ff\nn$ the Garside automorphism of $\BKL\nn$, \ie, the application defined on~$\BKL\nn$ by~$\ff\nn(\br)=\ddd\nn\,\br\,\ddd\nn\inv$.

An analog of the alternating normal form of the positive braid monoid $\BP\nn$ exists for the dual braid monoid~$\BKL\nn$: the \emph{rotating normal form} (see \cite{Fromentin2008,Fromentin2008a} for more details about this normal form). The rotating normal form is also based on an operation of splitting: the  $\ff\nn$-splitting.
Moreover, for each result on the alternating normal form used in this paper there exists a counterpart in the $\BKL\nn$context, see \cite{Fromentin2008} or \cite{Fromentin2008b} for more details.

A $\SBKL\nn$-word $\ww$ is said to be $\sigg$-definite if all the letters $\aa\indi\indii^\pm$ with highest $\indii$ appear only positively or only negatively.
This definition coincides with that given for $\SBP\nn$-words if we translate each letter $\aa\indi\indii$ of an $\SBKL\nn$-word to the $\SBP\nn$-word given in Definition~\ref{D:Aij}\,$(i)$.
Hence all the previous algortihms can be translated to the dual language, replacing~$\fl\nn$ by $\ff\nn$, $\Delta_\nn$ by $\ddd\nn$ and $\SBP\nn$ by~$\SBKL\nn$.
One of the advantage of the dual braid monoid is that its Garside element $\ddd\nn$ has length $\nno$, while $\Delta_\nn$ has length $\frac{n(n{-}1)}{2}$.
Therefore in the dual context, Algorithm~\ref{A:GarsideFraction} runing on $\ww$ returns a word $\ddd\nn^{-\tt}\,\uu$ whose length is at most $(2\nn{-}3)\lena{\wwt}$,where, for $\br\in\BB\nn$, $\lena{\br}$ denotes the word length of $\br$ with respect to $\SBKL\nn$ (as $\SBKL\nn$ contains $\SBP\nn$, we have necessary $\lena{\br}\le\lens{\br}$ for all braid~$\br$ of $\BB\nn$).
Hence Algorithm~\ref{A:Main} running on~$\ww$ returns a word of length at most $(2\nn{-}3)\lena{\wwt}$.

\bibliographystyle{ams-pln}
\bibliography{biblio}

\vspace{2em}

\noindent \textbf{Jean Fromentin}

Univ Lille Nord de France, F-59000 Lille, France

ULCO, LMPA J.~Liouville, B.P. 699, F-62228 Calais, France

CNRS, FR 2956, France

\url{fromentin@lmpa.univ-littoral.fr}

\vspace{2em}

\noindent\textbf{Luis Paris}

Universit\'e de Bourgogne

Institut de Math\'ematiques de Bourgogne, UMR~5584 du CNRS, 

B.P.~47870, 21078~Dijon~Cedex, France

\url{lparis@u-bourgogne.fr}

\end{document}